\documentclass[final]{amsart}

\usepackage{bm,nicefrac,cite}
\usepackage{tikz}
\usepackage{graphicx}
\usepackage{subfigure}
\usepackage[foot]{amsaddr}

\begin{document}

\author[R.~Baumann \and T.~P.~Wihler]{Ramona Baumann \and Thomas P.~Wihler} 
\address{Mathematics Institute, University of Bern, Switzerland}
\email{wihler@math.unibe.ch}

\title[A Nitsche FEM for Problems with Discontinuous Boundary Conditions]{A Nitsche Finite Element Approach for Elliptic Problems with Discontinuous Dirichlet Boundary Conditions}


\newcommand{\dd}{\,{\rm d}}
\newcommand{\ds}{\dd s}
\newcommand{\jl}{[\![}
\newcommand{\jr}{]\!]}
\newcommand{\jmp}[1]{\jl#1\jr}
\newcommand{\T}{{{\mathcal T}_h}}
\newcommand{\dx}{\dd\bm x}
\newcommand{\NN}[1]{\left\|#1\right\|}
\newcommand{\V}{\mathbb{V}(\T)}
\newcommand{\supp}{\mathrm{supp}}
\newcommand{\wh}[1]{\widehat{#1}}
\newcommand{\pw}{\mathsf{pw}}
\newcommand{\pts}[1]{\bm #1}
\newcommand{\set}[1]{\mathcal{#1}}
\newcommand{\A}{\set A}

\newtheorem{theorem}{Theorem}[section]
\newtheorem{lemma}[theorem]{Lemma}
\newtheorem{proposition}[theorem]{Proposition}
\newtheorem{corollary}[theorem]{Corollary}
\newtheorem{definition}[theorem]{Definition}
\newtheorem{example}[theorem]{Example}
\newtheorem{remrk}[theorem]{Remark}
\newenvironment{remark}[1]{\begin{remrk}\rm #1}{\end{remrk}}

\begin{abstract}
We present a numerical approximation method for linear diffusion-reaction problems with possibly discontinuous Dirichlet boundary conditions. The solution of such problems can be represented as a linear combination of explicitly known singular functions as well as of an $H^2$-regular part. The latter part is expressed in terms of an elliptic problem with regularized Dirichlet boundary conditions, and can be approximated by means of a Nitsche finite element approach. The discrete solution of the original problem is then defined by adding the singular part of the exact solution to the Nitsche approximation. In this way, the discrete solution can be shown to converge of second order with respect to the mesh size.
\end{abstract}


\maketitle

\section{Introduction}
Given a bounded, open and convex polygonal domain $\Omega\subset\mathbb{R}^2$ with straight
edges, we consider the linear diffusion-reaction problem
\begin{alignat}{2}
-\Delta u+\mu u&=f&\qquad&\text{in }\Omega,\label{eq:1a}\\
u&=g&&\text{on }\Gamma,\label{eq:1b}
\end{alignat}
where $\Gamma=\partial\Omega$ denotes the boundary of $\Omega$, $\,\mu\in
L^\infty(\Omega)$ is a nonnegative function, $f\in L^2(\Omega)$ is a source term, and $g\in L^2(\partial\Omega)$ is a possibly discontinuous function on~$\Gamma$ whose precise regularity will be specified later on.

Various formulations for~\eqref{eq:1a}--\eqref{eq:1b}, where the Dirichlet boundary data does not necessarily belong to~$H^{\nicefrac12}(\Gamma)$, exist in the literature. For instance, the \emph{very weak formulation} is based on twofold integration by parts of~\eqref{eq:1a} and, thereby, incorporates the Dirichlet boundary conditions in a natural way. It seeks a solution~$u\in L^2(\Omega)$ such that
\[
-\int_\Omega u\Delta v\dx +\int_{\Omega} \mu uv \dx =\int_\Omega fv\dx
-\int_\Gamma g\, \nabla v \cdot {\bf n}\ds
\]
for any~$v\in H^2(\Omega)\cap H_0^1(\Omega)$, where we write ${\bf n}$ for
the unit outward normal vector to the boundary $\Gamma$. Alternatively, the following saddle point formulation, which traces back to the work~\cite{Ne62}, may be applied: provided that~$g\in
H^{\nicefrac{1}{2}-\varepsilon}(\partial\Omega)$, for some
$\varepsilon\in[0,\nicefrac12)$, find~$u\in H^{1-\varepsilon}(\Omega)$
  with $u|_\Gamma=g$ such that
\begin{equation}\label{eq:saddle}
\int_\Omega\nabla u\cdot\nabla v\dx +\int_{\Omega} \mu uv \dx=\int_\Omega fv\dx
\end{equation}
for all~$v\in H^{1+\varepsilon}(\Omega)\cap H_0^1(\Omega)$; for results dealing with finite element approximations of~\eqref{eq:saddle}, we refer to~\cite{Ba71}. Another related approach is based on weighted Sobolev spaces (accounting for the local singularities of solutions with discontinuous boundary data), and has been analyzed in the context of $hp$-type discontinuous Galerkin methods in~\cite{HoustonWihler:12}.

The main idea of this paper is to represent the (weak) solution of~\eqref{eq:1a}--\eqref{eq:1b} in terms of a regular $H^2$ part as well as an explicitly known singular part (Section~\ref{sc:weak}). The latter is expressed by means of suitable singular functions which account for the local discontinuities in the Dirichlet boundary data (Section~\ref{sc:sing}). Here, it is crucial to ensure that the boundary data of the regular problem is sufficiently smooth as to provide an $H^2$ trace lifting (see Section~\ref{sc:trace}). We shall employ a classical Nitsche technique in order to discretize the regular part of the solution, and define the numerical approximation of~\eqref{eq:1a}--\eqref{eq:1b} by adding back the (exact) singular part (Section~\ref{sc:Nitsche}). A numerical experiment (Section~\ref{sc:num}) underlines that our approach provides optimally converging results.

Throughout the paper we shall use the following notation: For an open
domain~$\set D\subset\mathbb{R}^n$, $n\in\{1,2\}$, and~$p\in[1,\infty]$, we denote by~$L^p(\set D)$ the class of Lebesgue spaces on~$\set D$. For~$p=2$, we write~$\NN{\,\cdot\,}_{0,\set D}$ to signify the~$L^2$-norm on~$\set D$. Furthermore, for an integer $k\in\mathbb{N}_0$, we let $H^k(\set D)$ be the usual Sobolev space of order~$k$ on $\set D$, with norm~$\|\cdot\|_{k,D}$ and semi-norm~$|\cdot|_{k,\set D}$. The set~$H^1_0(\set D)$ represents the subspace of~$H^1(\set D)$ of all functions with zero trace along~$\partial\set D$. If~$\set D$ is represented as a (disjoint) finite union of open sets, that is, $\overline{\set D}=\bigcup_i\overline{\set D}_i$, and~$X$ is any class of function spaces, then we write~$X_{\pw}(\set D)=\Pi_{i}X(D_i)$ to mean the set of all functions which belong \emph{piecewise} (with respect to the partition~$\{\set D_i\}_i$) to~$X$.

\section{Problem formulation}

The aim of this section is to establish a suitable framework for the weak solution of~\eqref{eq:1a}--\eqref{eq:1b}.

\subsection{Notation}
Let $\set A=\{\pts A_i\}_{i=1}^M\subset\partial\Omega$, with $\pts A_i\neq\pts  A_j$, for~$1\le i\neq j\le M$, be a finite set of points on the boundary of the polygonal domain $\Omega$, which are numbered in counter-clockwise
direction along~$\partial\Omega$; the points in~$\A$ mark the
locations where the Dirichlet boundary condition~$g$ from~\eqref{eq:1b} exhibits discontinuities. Furthermore, we denote by~$\Gamma_i\subset\Gamma$, $i=1,2,\ldots,M$, the open edge which connects the two points~$\pts A_i$ and~$\pts A_{i+1}$; in the sequel, we shall identify indices~$0\simeq M$, $1\simeq M+1$, etc.; for instance, we have $\pts A_{M+1}=\pts A_1$ and~$\pts A_0=\pts A_M$, or~$\Gamma_{M+1}=\Gamma_1$ and~$\Gamma_0=\Gamma_M$, etc. Moreover, let~$\omega_i\in(0,\pi]$ signify the interior angle of~$\Omega$ at~$\pts A_i$ (in counter-clockwise direction). Finally, for~$\phi\in C^0_{\pw}(\Gamma)$, i.e., $\phi|_{\Gamma_i}\in C^0(\Gamma_i)$, for~$1\le i\le M$, we set~$\phi_i:=\phi|_{\Gamma_i}$, and define the one-sided limits
\[
\phi(\pts A_i^+)=\lim_{\genfrac{}{}{0pt}{}{\bm x\to\bm A_i}{\bm x\in\Gamma_{i}}}\phi_i(\bm x),\qquad
\phi(\pts A_i^-)=\lim_{\genfrac{}{}{0pt}{}{\bm x\to\bm A_i}{\bm x\in\Gamma_{i-1}}}\phi_{i-1}(\bm x),
\]
and the jumps~$\jmp{\phi}_i=\phi(\pts A_i^+)-\phi(\pts A_i^-)$, for $i=1,\ldots, M$.

\subsection{Singular functions}\label{sc:sing}
In the following, based on the partition~$\overline{\Gamma}=\bigcup_{i=1}^M\overline{\Gamma}_i$, we assume that the boundary data $g$ from~\eqref{eq:1b} satisfies
\begin{equation}\label{eq:gpw}
g\in H_{\pw}^2(\Gamma),
\end{equation} 
i.e.,  with the notation above, we have $g_i\in H^2(\Gamma_i)$, for~$1\le i\le M$. We note the continuous Sobolev embedding~$H^1(\Gamma_i)\hookrightarrow L^\infty(\Gamma_i)$, i.e.,
\begin{equation}\label{eq:Sobolev}
\sup_{\pts x\in\Gamma_i}|v(\pts x)|\le C\|v\|_{1,\Gamma_i},\qquad\forall v\in H^1(\Gamma_i),\quad i=1,\ldots, M,
\end{equation}
for a constant~$C=C(\Gamma_i)>0$. In particular, this implies that the values of $g(\pts A_i^\pm)$ and~$g'(\pts A_i^\pm)$, with~$g'$ denoting the (edgewise) tangential derivative of~$g$ in counter-clockwise direction along~$\Gamma$, are well-defined. Hence, for~$r_i\neq 0$, we may consider the singular functions (cf.~\cite[Lemma~6.1.1]{MelenkDiss}), for $1\le i\le M$,
\begin{equation}\label{eq:Theta}
\Theta_i(r_i,\theta_i)=
\begin{cases}\displaystyle
g(\pts A_i^+)-\frac{\theta_i}{\omega_i}\jmp{g}_i &\text{if }\omega_i\in(0,\pi),\\[2ex] \displaystyle
g(\pts A_i^+)-\frac{1}{\pi}\left(\theta_i\jmp{g}_i+\sigma_i(r_i,\theta_i)\jmp{g'}_i\right) &\text{if }\omega_i=\pi,
\end{cases}
\end{equation}
with
\[
\sigma_i(r_i,\theta_i)=r_i\left(\ln(r_i)\sin(\theta_i)+\theta_i\cos(\theta_i)\right).
\]
Here, $(r_i,\theta_i)$ denote polar coordinates with 
respect to a local coordinate system centered at $\pts A_i$ such that $\theta_i=0$ on $\Gamma_i$, and $\theta_i=\omega_i$ on $\Gamma_{i-1}$. 
We note that $\Theta_i$ is harmonic away from~$\pts A_i$, i.e., $\Delta\Theta_i =0$ in $\Omega$. Since~$\Theta_i$ is smooth away from~$\pts A_i$, there holds
\begin{equation}\label{eq:gjmp}
\jmp{\Theta_i}_j=\delta_{ij}\jmp{g}_i,\qquad 1\le i,j\le M,
\end{equation}
where~$\delta_{ij}$ is Kronecker's delta. In addition, for~$\omega_j=\pi$, we have
\begin{equation}\label{eq:g'jmp}
\jmp{\Theta_i'}_j=\delta_{ij}\jmp{g'}_j,
\end{equation} 
for~$i=1,\ldots,M$.

\subsection{Trace lifting}\label{sc:trace}

Defining the function
\begin{equation}\label{eq:ghat}
\widehat g:\,\Gamma\to\mathbb{R},\qquad \widehat g:=g-\sum_{i=1}^M\Theta_i|_\Gamma,
\end{equation}
with~$\Theta_i$ from~\eqref{eq:Theta}, and recalling~\eqref{eq:gjmp}, we note that
\begin{equation}\label{eq:id1}
\jmp{\widehat g}_j=\jmp{g}_j-\sum_{i=1}^M\jmp{\Theta_i}_j=0,\qquad 1\le j\le M,
\end{equation}
i.e., $\widehat g$ is continuous along the boundary~$\Gamma$. Similarly, whenever~$\omega_j=\pi$, using~\eqref{eq:g'jmp}, we have
\begin{equation}\label{eq:id2}
\jmp{\widehat g'}_j=\jmp{g'}_j-\sum_{i=1}^M\jmp{\Theta_i'}_j=0.
\end{equation} 

\begin{lemma}\label{lm:g}
There holds the estimate
\[
\sum_{i=1}^M\|\widehat g_i\|_{2,\Gamma_i}\le C\sum_{i=1}^M\|g_i\|_{2,\Gamma_i},
\]
where~$C>0$ is a constant independent of~$g$.
\end{lemma}

\begin{proof}
By definition of~$\widehat g$, see~\eqref{eq:ghat}, for any~$1\le i\le M$, there holds
\begin{align*}
\|\widehat g_i\|_{2,\Gamma_i}
&\le \|g_i\|_{2,\Gamma_i}+\sum_{j=1}^M\|\Theta_j\|_{2,\Gamma_i}.
\end{align*}
Since~$\Theta_j$ is a linear function along both~$\Gamma_{j-1}$ and~$\Gamma_{j}$ and smooth on~$\bigcup_{k\neq j-1,j}\overline{\Gamma}_k$, we deduce the bound
\[
\|\Theta_j\|_{2,\Gamma_i}
\le C_{ij}\left(|g(\pts A_j^+)|+\omega_j^{-1}|\jmp{g}_j|+|\jmp{g'}_j|\right),
\]
where~$C_{ij}>0$ is a constant depending on~$\pts A_j$ and~$\Gamma_i$. Hence,
\begin{align*}
\sum_{i=1}^M\|\widehat g_i\|_{2,\Gamma_i}
&\le \sum_{i=1}^M\|g_i\|_{2,\Gamma_i}+\sum_{i,j=1}^MC_{ij}\left(|g(\pts A_j^+)|+\omega_j^{-1}|\jmp{g}_j|+|\jmp{g'}_j|\right)\\
&\le \sum_{i=1}^M\|g_i\|_{2,\Gamma_i}+C\sum_{i=1}^M\left(\|g_i\|_{\infty,\Gamma_i}+\|g'\|_{\infty,\Gamma_i}\right).
\end{align*}
Using~\eqref{eq:Sobolev}, the proof is complete.
\end{proof}

The identities~\eqref{eq:id1} and~\eqref{eq:id2} together with the previous lemma imply the following result.

\begin{lemma}\label{lifting}
There exists a lifting $\widehat{U}\in H^2(\Omega)$ of the boundary data 
$\widehat g$, i.e., $\widehat{U}|_{\Gamma}=\widehat{g}$ in the sense of traces, with
\begin{equation}\label{ub:lifting}
\|\widehat{U}\|_{2,\Omega}\leq C\sum_{i=1}^M\| g_i\|_{2,\Gamma_i},
\end{equation}
where~$C>0$ is a constant independent of~$g$.
\end{lemma}

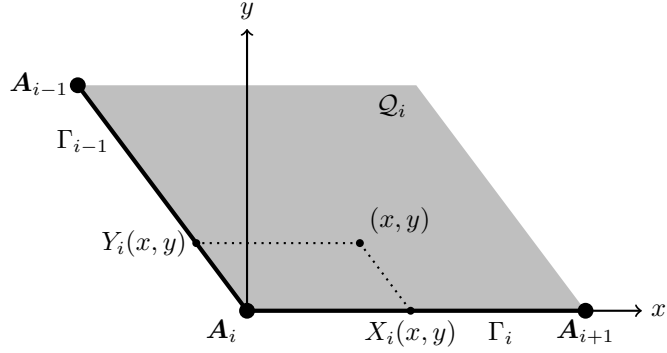
\begin{figure}
\begin{center}
\begin{tikzpicture}[scale=1.5]
\draw[fill=gray!50,draw opacity=0] (0,0) -- (3,0) -- (1.5,2) -- (-1.5,2) -- cycle;
\draw (1.5,2) node [anchor=north east] {$\set Q_i$};
\draw[->,thick] (0,0) -- (3.5,0);
\draw (3.5,0) node [right] {$x$}; 
\draw[->,thick] (0,0) -- (0,2.5);
\draw (0,2.5) node [above] {$y$};
\draw[ultra thick] (0,0) -- (3,0) node[near end,below] {$\Gamma_i$};
\draw[ultra thick] (0,0) -- (-1.5,2) node[near end,left] {$\Gamma_{i-1}$};
\draw (0,0) node [anchor = north east]{$\pts A_i$};
\fill (0,0) circle (2pt);
\draw (-1.5,2) node [anchor = east]{$\pts A_{i-1}$};
\fill (-1.5,2) circle (2pt);
\draw (3,0) node [anchor = north]{$\pts A_{i+1}$};
\fill (3,0) circle (2pt);
\draw (1,0.6) node [anchor = south west] {$(x,y)$};
\fill (1,0.6) circle (1pt);
\draw [thick,dotted] (1,0.6) -- (-0.45,0.6);
\draw [thick,dotted] (1,0.6) -- (1.45,0);
\draw (1.45,0) node [anchor = north] {$X_i(x,y)$};
\draw (-0.45,0.6) node [anchor = east] {$Y_i(x,y)$};
\fill (1.45,0) circle (1pt);
\fill (-0.45,0.6) circle (1pt);
\end{tikzpicture}
\end{center}
\caption{Graphical illustration of (local) trace lifting construction.}
\label{fig:lifting}
\end{figure}

\begin{proof}
We use a partition of unity approach. Specifically, to each corner~$\pts A_i$ of~$\Omega$, we associate a function~$\phi_i\in C^\infty(\overline\Omega)$ such that $\sum_{i=1}^M\phi_i(\pts x)=1$ for any $\pts x\in\Gamma$, and~$\supp(\phi_i)\cap\Gamma\subset\Gamma_{i-1}\cup\{\pts A_i\}\cup\Gamma_i$, for~$1\le i\le M$. 

Fix~$i\in\{1,\ldots, M\}$. If~$0<\omega_i<\pi$, we may assume, without loss of generality, that~$\pts A_i$ coincides with the origin~$(0,0)$, and the edge~$\Gamma_i$ can be placed on the first coordinate axis. Denoting the (Cartesian) coordinates in this system by~$(x,y)$, we let
\[
X_i(x,y)=\left(x-\frac{y}{\tan(\omega_i)},0\right),\qquad
Y_i(x,y)=\left(\frac{y}{\tan(\omega_i)},y\right);
\]
see Figure~\ref{fig:lifting} for a graphical illustration. Observe that
\begin{align*}
X_i|_{\Gamma_{i-1}}&=\pts 0, &X_i|_{\Gamma_i}&=\mathsf{id},\\
Y_i|_{\Gamma_{i-1}}&=\mathsf{id}, &Y_i|_{\Gamma_{i}}&=\pts 0,
\end{align*}
where~$\mathsf{id}$ is the identity function. Then, for $(x,y)\in\Omega$, we define the lifting
\[
\widehat U_i=\begin{cases}
\left(\widehat g|_{\Gamma_{i}}\circ X_i+\widehat g|_{\Gamma_{i-1}}\circ Y_i-\widehat g(\pts A_i)\right)\phi_i&\text{in }\set Q_i\cap\Omega,\\
0&\text{on }\Omega\setminus\set Q_i,
\end{cases}
\]
where
\[
\set Q_i=\left\{\bm x=\omega_1(\pts A_{i+1}-\pts A_i)+\omega_2(\pts A_{i-1}-\pts A_{i}):\,\omega_1,\omega_2\in(0,1)\right\};
\]
cf. the gray area in Figure~\ref{fig:lifting}. The lifting~$\widehat U_i$ satisfies the boundary condition
\begin{equation}\label{eq:traceU}
\widehat U_i|_{\Gamma}=\widehat g\phi_i|_{\Gamma}.
\end{equation} 
Furthermore, we note that
\[
\|\widehat U_i\|_{2,\Omega}\le C\left(\|\widehat g_{i-1}\|_{2,\Gamma_{i-1}}+\|\widehat g_i\|_{2,\Gamma_i}+|\widehat g(\pts A_i)|\right).
\]
Using~\eqref{eq:Sobolev}, we obtain
\begin{equation}\label{eq:H2U}
\|\widehat U_i\|_{2,\Omega}\le C\left(\|\widehat g_{i-1}\|_{2,\Gamma_{i-1}}+\|\widehat g_i\|_{2,\Gamma_i}\right).
\end{equation}
If~$\omega_i=\pi$, then the function~$\widehat g\phi_i|_{\Gamma}$ belongs to~$H^{\nicefrac{3}{2}}(\Gamma)$, and by the trace theorem, there exists~$\widehat U_i\in H^2(\Omega)$ which again satisfies~\eqref{eq:traceU} as well as~\eqref{eq:H2U}. Therefore, letting
\[
\widehat U=\sum_{i=1}^M\widehat U_i,
\]
we see that~$\widehat U|_\Gamma=\widehat g$, and
\[
\|\widehat U\|_{2,\Omega}\le\sum_{i=1}^M\|\widehat U_i\|_{2,\Omega}\le C\sum_{i=1}^M\|\widehat g_i\|_{2,\Gamma_i}.
\]
Employing Lemma~\ref{lm:g} completes the argument.
\end{proof}

\subsection{Weak solution}\label{sc:weak}

Let 
\begin{equation}\label{eq:fhat}
\widehat f:= f-\mu \sum_{i=1}^M\Theta_i\in L^2(\Omega).
\end{equation}
Then, proceeding analogously as in the proof of Lemma~\ref{lm:g}, we deduce that
\begin{equation}\label{eq:fhatbound}
\|\widehat f\|_{0,\Omega}\le \|f\|_{0,\Omega}+\mu\sum_{i=1}^M\|\Theta_i\|_{0,\Omega}
\le \|f\|_{0,\Omega}+C\sum_{i=1}^M\|g_i\|_{2,\Gamma_i},
\end{equation}
with a constant independent of~$f$ and~$g$. Consider the regularized problem
\begin{alignat}{2}
-\Delta \widehat u+\mu \widehat u&=\widehat f&\qquad&\text{in }\Omega,\label{eq:uhat}\\
\widehat u&=\widehat g&&\text{on }\Gamma,\label{eq:bcuhat}
\end{alignat}
where~$\widehat g$ is the boundary function from~\eqref{eq:ghat}.

\begin{proposition}
Let $\Omega$ be a convex and bounded polygonal domain. Then, there exists a unique solution $\widehat u\in H^2(\Omega)$ to ~\eqref{eq:uhat}--\eqref{eq:bcuhat} that satisfies the stability bound
\begin{equation}
\|\widehat{u}\|_{2,\Omega}\leq C\left(\|f\|_{0,\Omega}+\sum_{i=1}^M\| g_i\|_{2,\Gamma_i}\right)\label{ub:uhat},
\end{equation}
with a constant~$C>0$ depending on~$\Omega$, and on~$\mu$.
\end{proposition}

\begin{proof}
Proposition \ref{lifting} provides the existence of a function $\widehat{U}\in H^2(\Omega)$ with 
$\widehat{U}\big|_{\Gamma}=\widehat{g}$. Since $\widehat{f}+\Delta\widehat{U}-\mu\widehat{U}$ belongs to $L^2(\Omega)$, elliptic regularity theory in convex polygons (see, e.g., \cite{Gr85,BG88,Da88}) implies the existence of a unique remainder function $\widehat\rho\in H^2(\Omega)$ with
\begin{alignat*}{2}
-\Delta \widehat\rho+\mu \widehat\rho&=\widehat{f}+\Delta\widehat{U}-\mu\widehat{U}&\qquad&\text{in }\Omega,\\
\widehat\rho&=0&&\text{on }\Gamma,
\end{alignat*}
and
\begin{equation}
\|\widehat\rho\|_{2,\Omega}\leq C\|\widehat{f}+\Delta\widehat{U}-\mu\widehat{U} \|_{0,\Omega}
\le C\left(\|\widehat{U}\|_{2,\Omega}+\|\widehat{f}\|_{0,\Omega}\right)
\label{ub:u}.
\end{equation}
Thus, the function $\widehat{u}:=\widehat{U}+\widehat\rho$ belongs to~$H^2(\Omega)$. Furthermore, it holds that
\[
-\Delta \widehat{u}+\mu \widehat{u}=-\Delta\widehat{U}+\mu\widehat{U}-\Delta\widehat\rho+\mu\widehat\rho=\widehat{f} \qquad\text{in }\Omega,
\]
as well as 
\[
\widehat{u}\big|_{\Gamma}=\widehat{U}\big|_{\Gamma}+\widehat\rho\big|_{\Gamma}=\widehat{g}.
\] 
In addition, combining~\eqref{ub:lifting} and~\eqref{ub:u} yields 
\begin{align*}
\|\widehat{u}\|_{2,\Omega}
&\le\|\widehat{U}\|_{2,\Omega}+\|\widehat\rho\|_{2,\Omega}
\le C\left(\|\widehat{U}\|_{2,\Omega}+\|\widehat{f}\|_{0,\Omega}\right)\\
&\le C\left(\sum_{i=1}^M\| g_i\|_{2,\Gamma_i}+\|\widehat{f}\|_{0,\Omega}\right),
\end{align*}
which, by virtue of~\eqref{eq:fhatbound}, results in the bound~\eqref{ub:uhat}.
\end{proof}

\begin{definition}
We call the function $u$ defined by
\begin{equation}\label{def:u}
u:=\widehat u+\sum_{i=1}^M\Theta_i,
\end{equation}
with $\widehat u$ the unique $H^2$-solution of~\eqref{eq:uhat}--\eqref{eq:bcuhat}, the \emph{weak solution} of~\eqref{eq:1a}--\eqref{eq:1b}.
\end{definition}


\begin{remark}
It can be verified easily that the weak solution defined in~\eqref{def:u} belongs to a class of weighted Sobolev spaces; cf., e.g., \cite{BG88,BG89}. The norms of these spaces contain local radial weights at the discontinuity points~$\set A$ of the Dirichlet boundary data, and, thereby, account for possible singularities in the solution of~\eqref{eq:1a}--\eqref{eq:1b}. Based on an inf-sup theory, the work~\cite{HoustonWihler:12} shows that~\eqref{eq:1a}--\eqref{eq:1b} exhibits a unique solution within this framework. 
\end{remark}


\section{Numerical approximation}\label{sc:numerics}

The purpose of this section is to discretize~\eqref{eq:1a}--\eqref{eq:1b} by a finite element approach. Specifically, we will employ a Nitsche method to obtain a numerical approximation of the elliptic problem~\eqref{eq:uhat}, with the possibly non-homogeneous Dirichlet boundary condition~\eqref{eq:bcuhat}. The discrete solution will then be defined similarly as in~\eqref{def:u}.

\subsection{Meshes and spaces}

We consider regular, quasi-uniform meshes ${\mathcal T}_h$ of mesh size~$h>0$, which partition~$\Omega\subset\mathbb{R}^2$ into open disjoint triangles
and/or parallelograms $\{K\}_{K\in{\mathcal T}_h}$, i.e.,
$\overline\Omega=\bigcup_{K\in\T}\overline K$. Each element
$K\in{\mathcal T}_h$ is an affinely mapped image of the reference
triangle~$\wh T=\{(\wh x,\wh y):\,-1<\wh x<1,-1<\wh y<-\wh x\}$ or the
reference square $\widehat{S}=(-1,1)^2$, respectively. Moreover, we define the conforming finite element space
\[
\V=\{v\in H^1(\Omega) : v|_K \in{\mathbb{S}}(K), K\in {\T}\},
\]
where, for~$K\in\T$, we write $\mathbb{S}(K)$ to mean either
the space~$\mathbb{P}_{1}(K)$ of all polynomials of total degree at
most~$1$ on~$K$ or the space~$\mathbb{Q}_{1}(K)$ of all
polynomials of degree at most~$1$ in each coordinate direction on
$K$.

\subsection{Nitsche discretization}\label{sc:Nitsche}

The classical Nitsche approach~\cite{N71} for the numerical approximation of~\eqref{eq:uhat}--\eqref{eq:bcuhat} is given by finding~$\widehat u_h\in\V$ such that
\begin{equation}\label{eq:Nitsche}
a_h(\widehat u_h,v)=l_h(v) \quad \text{for all }v\in \V.
\end{equation}
Here, denoting by~$\nabla_h$ the elementwise gradient operator, we define the bilinear form
\begin{align*}
a_h(w,v)&=\int_{\Omega}\left\{\nabla_h w \cdot\nabla_h v+\mu wv\right\}\dx\\
&\quad							-\int_{\partial \Omega}v\left(\nabla_h w\cdot {\bf n}\right)\ds
							-\int_{\partial \Omega}w\left(\nabla_h v \cdot {\bf n}\right)\ds
							+\frac{\gamma}{h}\int_{\partial \Omega}w v\ds,
\end{align*}
as well as the linear functional
\[
l_h(v)=\int_{\Omega}\widehat f v\dx
				-\int_{\partial \Omega}\widehat g\left(\nabla_h v\cdot {\bf n}\right)\ds
				+\frac{\gamma}{h}\int_{\partial\Omega}\widehat g v\ds,
\]
with~$\widehat g$ and~$\widehat f$ from~\eqref{eq:ghat} and~\eqref{eq:fhat}, respectively.
The penalty parameter~$\gamma>0$ appearing in both forms is chosen sufficiently large (but independent of the mesh size) as to guarantee the well-posedness of the weak formulation~\eqref{eq:Nitsche}; this can be shown in a similar way as in the context of discontinuous Galerkin methods; see, e.g., ~\cite{ABCM}. In addition, referring to~\cite[Satz~2]{N71}, cf.~also~\cite[Section~5.1]{ABCM}, there holds the \emph{a priori} error estimate
\begin{equation}\label{eq:errNitsche}
\|\widehat u-\widehat u_h\|_{0,\Omega}\le Ch^2|\widehat u|_{2,\Omega},
\end{equation}
with a constant~$C=C(\mu,\widehat f,\widehat g)>0$ independent of the mesh size~$h$.

\begin{definition}
Analogously to \eqref{def:u}, we define the discrete solution of~\eqref{eq:1a}--\eqref{eq:1b} by
\begin{equation}\label{def:uh}
u_h:=\widehat u_h+\sum_{i=1}^M\Theta_i,
\end{equation}
where~$\widehat u_h\in\V$ is the Nitsche solution from~\eqref{eq:Nitsche}, and~$\{\Theta_i\}_{i=1}^M$ are the singular functions from~\eqref{eq:Theta}.
\end{definition}

\begin{theorem}\label{ee:apriori}
Let $u$ be the solution of~\eqref{eq:1a}--\eqref{eq:1b} given by~\eqref{def:u}, and $u_h$ its discrete counterpart from~\eqref{def:uh}. Then, there holds the \emph{a priori} error estimate
\begin{equation}\label{eq:apriori}
\|u-u_h\|_{0,\Omega}\le Ch^2,
\end{equation}
with a constant~$C=C(\mu,f,g)>0$ independent of~$h$.
\end{theorem}

\begin{proof}
We recall the definitions~\eqref{def:u} and~\eqref{def:uh} in order to notice
\[
u-u_h=\widehat u+\sum_{i=1}^M\Theta_i-\left(\widehat u_h+\sum_{i=1}^M\Theta_i\right)
=\widehat u-\widehat u_h.
\]
Therefore, applying~\eqref{eq:errNitsche} yields
\[
\|u-u_h\|_{0,\Omega}\le Ch^2|\widehat u|_{2,\Omega}.
\]
Finally, recalling~\eqref{ub:uhat} completes the proof.
\end{proof}

\subsection{Numerical example}\label{sc:num}
On the rectangle $\Omega=(-1,1)\times(0,1)$ we consider the elliptic boundary value problem
\begin{alignat*}{2}
-\Delta u+u&=e^{-r^2}(5-4r^2)\theta &\qquad&\text{in }\Omega\\
u&=g&&\text{on }\Gamma,
\end{alignat*}
with the Dirichlet boundary data~$g$ chosen such that the analytical solution is given by
\[
u(r,\theta)=e^{-r^2}\theta.
\]
Here, $(r,\theta)$ denote polar coordinates in $\mathbb{R}^2$. Note that the solution~$u$ is smooth along $\Gamma$ except at the origin, where it exhibits a discontinuity jump. In particular, it follows that~$u\not\in H^1(\Omega)$.

Starting from a regular coarse mesh, we investigate the practical performance of the a priori error estimate derived in Theorem \ref{ee:apriori} within a sequence of uniformly refined $\mathbb{P}_1$ elements. In Figure \ref{fig:L2errorplot} we present a comparison of the $L^2$ norm of the error versus the mesh size $h$ on a log-log scale for each of the meshes. Our results are in line with the \emph{a priori} error estimate~\eqref{eq:apriori}, and show that the discrete solution~$u_h$ from~\eqref{def:uh} converges of second order with respect to the mesh size~$h$. Moreover, in Figure \ref{fig:plot} we show the Nitsche solution $\widehat u_h\in\V$ defined in~\eqref{eq:Nitsche}, as well as the computed solution $u_h$ for a mesh consisting of 1024 elements.

\begin{figure}
\centering
\includegraphics[width=0.6\textwidth]{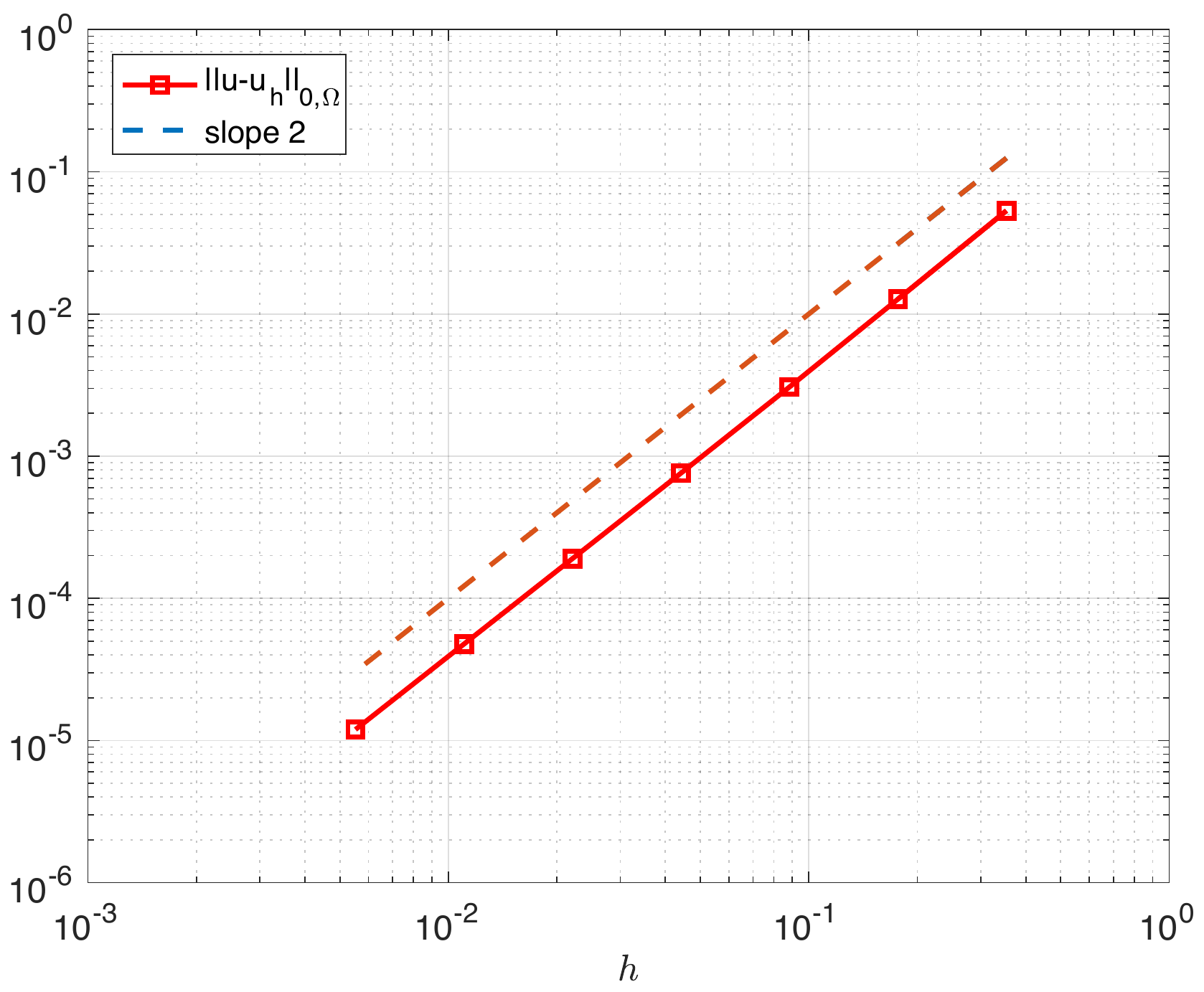}
\caption{$L^2$ error $\|u-u_h\|_{0,\Omega}$ against mesh size $h$ compared to a reference line with slope~$2$ (expected behaviour).}
\label{fig:L2errorplot}
\end{figure}

\begin{figure} 
	\centering
	\includegraphics[width=0.45\linewidth]{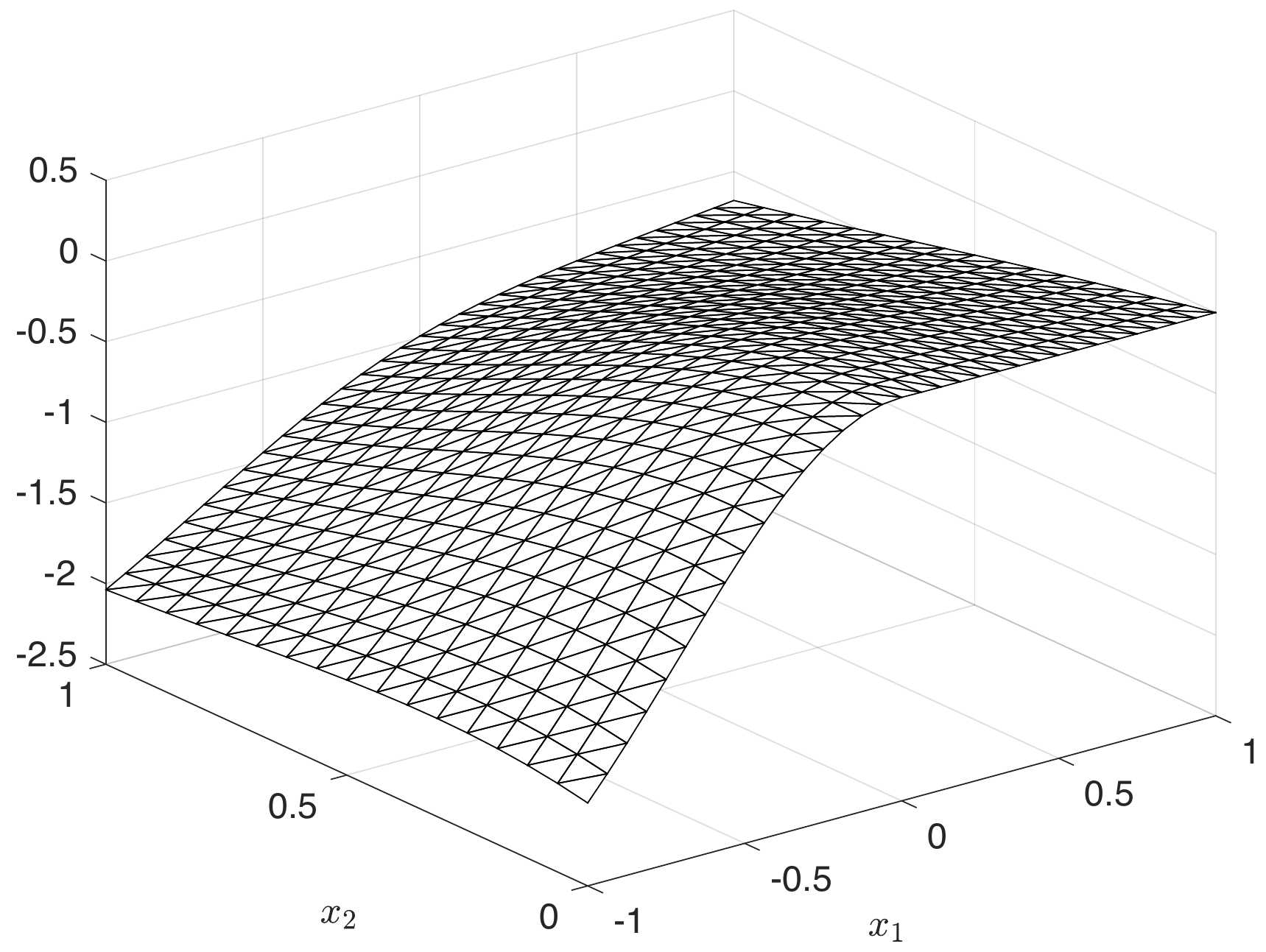}\quad
	\includegraphics[width=0.45\linewidth]{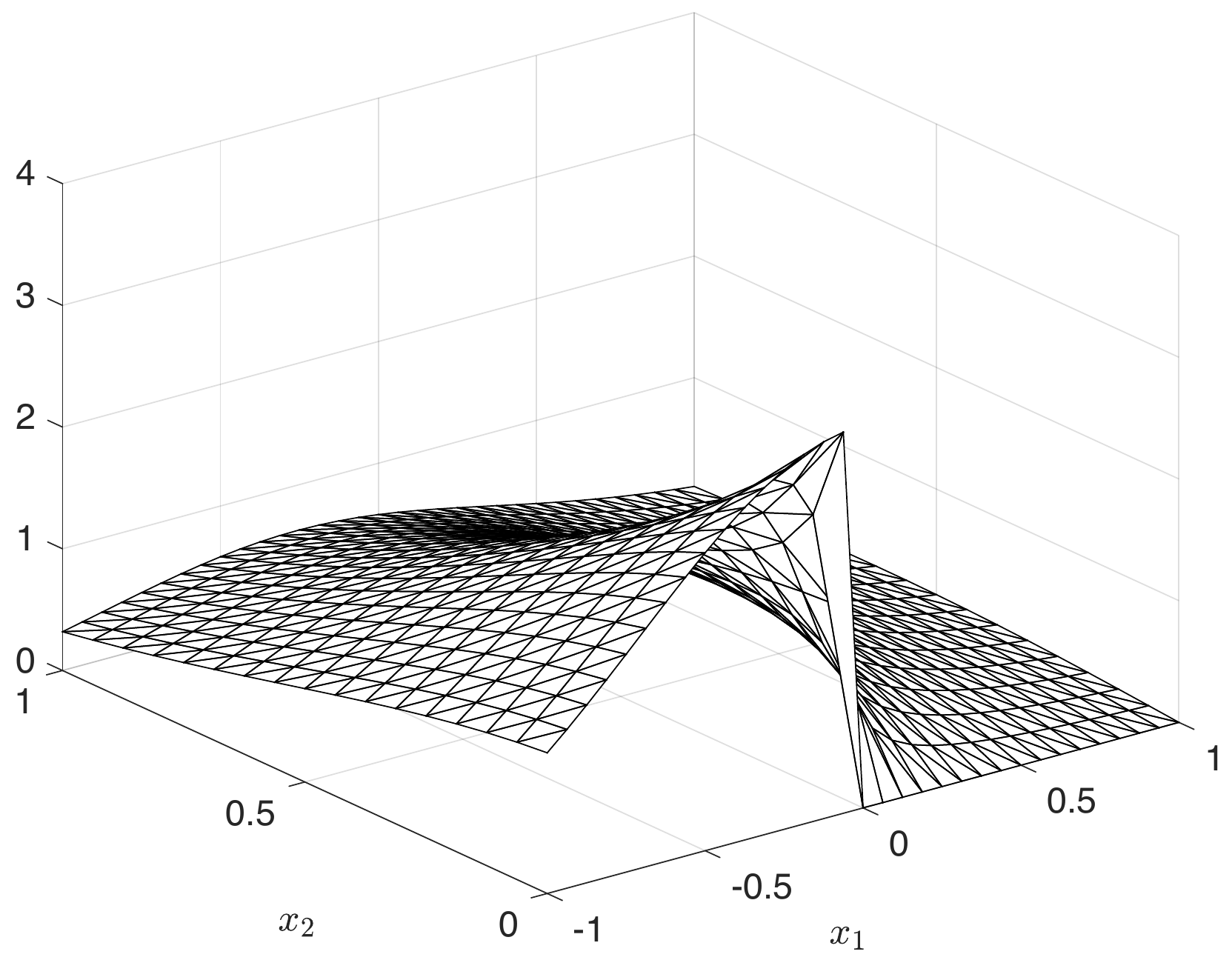}
	\caption{Nitsche solution (left) and discrete solution (right) 	based on a uniform mesh with $1024$ elements.} 
	\label{fig:plot}		
\end{figure}

\bibliographystyle{amsplain}
\bibliography{literature}

\providecommand{\bysame}{\leavevmode\hbox to3em{\hrulefill}\thinspace}
\providecommand{\MR}{\relax\ifhmode\unskip\space\fi MR }
\providecommand{\MRhref}[2]{%
  \href{http://www.ams.org/mathscinet-getitem?mr=#1}{#2}
}
\providecommand{\href}[2]{#2}
\begin{thebibliography}{10}

\bibitem{ABCM}
D.~N. Arnold, F.~Brezzi, B.~Cockburn, and L.~D. Marini, \emph{Unified analysis
  of discontinuous {G}alerkin methods for elliptic problems}, SIAM J. Numer.
  Anal. \textbf{39} (2001), 1749--1779.

\bibitem{BG88}
I.~Babu{\v s}ka and B.~Q. Guo, \emph{Regularity of the solution of elliptic
  problems with piecewise analytic data. {I}. {Boundary value problems for
  linear elliptic equation of second order}}, SIAM J. Math. Anal. \textbf{19}
  (1988), 172--203.

\bibitem{BG89}
I.~Babu{\v{s}}ka and B.~Q. Guo, \emph{Regularity of the solution of elliptic
  problems with piecewise analytic data. {II}. {T}he trace spaces and
  application to the boundary value problems with nonhomogeneous boundary
  conditions}, SIAM J. Math. Anal. \textbf{20} (1989), no.~4, 763--781.

\bibitem{Ba71}
I.~Babu\v{s}ka, \emph{Error bounds for finite element method}, Numer. Math.
  \textbf{16} (1971), no.~4, 322--333.

\bibitem{Da88}
M.~Dauge, \emph{Elliptic boundary value problems on corner domains}, Lecture
  Notes in Mathematics, no. 1341, Springer-Verlag, 1988.

\bibitem{Gr85}
P.~Grisvard, \emph{Elliptic problems in nonsmooth domains}, Monographs and
  Studies in Mathematics, vol.~24, Pitman (Advanced Publishing Program),
  Boston, MA, 1985.

\bibitem{HoustonWihler:12}
P.~Houston and T.~P. Wihler, \emph{Second-order elliptic {PDE}s with
  discontinuous boundary data}, IMA J. Numer. Anal. \textbf{32} (2012), no.~1,
  48--74.

\bibitem{MelenkDiss}
J.~M. Melenk, \emph{On generalized finite element methods}, Ph.D. thesis,
  University of Maryland, 1995.

\bibitem{Ne62}
J.~Ne\v{c}as, \emph{Sur une m\'ethode pour r\'esoudre les \'equations aux
  d\'eriv\'ees partielles du type elliptique, voisine de la variationelle},
  Ann.~Scuola Norm.~Sup., Pisa \textbf{16} (1962), 305--326.

\bibitem{N71}
J.~Nitsche, \emph{{\"U}ber ein {V}ariationsprinzip zur {L}{\"o}sung von
  {D}irichlet {P}roblemen bei {V}erwendung von {T}eilr{\"a}umen, die keinen
  {R}andbedingungen unterworfen sind}, Abh. Math. Sem. Univ. Hamburg
  \textbf{36} (1971), 9--15.

\end{thebibliography}

\end{document}